\documentclass{amsart}

\usepackage[T1]{fontenc}
\usepackage{lmodern}
\usepackage{tikz-cd}
\usepackage[hyphens]{url}
\usepackage{amsthm,amsmath,amssymb,amsfonts,dsfont}
\usepackage{mathtools}
\usepackage[margin=1.5in]{geometry}
\usepackage{xcolor}
\usepackage{microtype}
\usepackage[colorlinks,linkcolor=blue,citecolor=blue,urlcolor=blue]{hyperref}
\usepackage[indentafter]{titlesec}
\usepackage{mathrsfs}
\usepackage{bm}

\titleformat{name=\section}{}{\thetitle.}{0.8em}{\centering\scshape}
\titleformat{\subsection}[runin]
  {\normalfont\bfseries}{\thesubsection}{1em}{}
\numberwithin{equation}{section}
\makeatletter
\def\l@section{\@tocline{1}{0pt}{1pc}{1.5pc}{}}
\makeatother

\newtheorem{definition}{\textbf{Definition}}[section]

\newtheorem{corollary}[definition]{\textbf{Corollary}}
\newtheorem{lemma}[definition]{\textbf{Lemma}}
\newtheorem{proposition}[definition]{\textbf{Proposition}}

\newtheorem{remark}[definition]{\textbf{Remark}}
\newtheorem{letterthm}{Theorem}

\newtheorem*{conjecture T}{\textbf{Connes' rigidity conjecture for ICC
property (T) groups}}
\newtheorem*{conjecture HRL}{\textbf{Connes' rigidity conjecture for higher rank lattices}}

\DeclareMathOperator{\SL}{SL}
\DeclareMathOperator{\Sp}{Sp}
\DeclareMathOperator{\GL}{GL}
\DeclareMathOperator{\EL}{EL}
\DeclareMathOperator{\Hom}{Hom}

\DeclareMathOperator{\Ann}{Ann}
\newcommand{\Z}{\mathbb Z}
\newcommand{\VN}{ {L}}
\newcommand{\calR}{\mathcal R}

\newcommand{\norm}[1]{\lVert #1\rVert}
\newcommand{\ip}[2]{\langle #1,#2\rangle}

\def\car{\curvearrowright}

\def\CP{\mathcal{P}}

\def\d{\mathrm{d}}

\title[ICC Property~(T) Groups without W$^*$-Superrigidity]
{ICC Property~(T) Groups without W$^*$-Superrigidity}
\author{Shuoxing Zhou}
\address{\'Ecole Normale Sup\'erieure\\
D\'epartement de math\'ematiques et applications\\
45 rue d'Ulm\\
75230 Paris Cedex 05\\
FRANCE}
\email{shuoxing.zhou@ens.psl.eu}

\begin{document}

\begin{abstract}
We construct two explicit countable discrete groups $\Gamma_1$ and $\Gamma_2$ that both have infinite conjugacy classes (ICC) and Kazhdan's property~(T). Although $\Gamma_1$ and $\Gamma_2$ are not isomorphic as groups, their group von Neumann algebras are isomorphic: $\VN(\Gamma_1)\cong\VN(\Gamma_2)$. This provides a counterexample to Connes' rigidity conjecture for ICC property~(T) groups. This result was obtained with assistance from GPT-5.6 Sol, independently of and concurrently with work by OpenAI \cite{OpenAI26}.
\end{abstract}

\maketitle

\section{Introduction}

In 1980, Alain Connes proved that if $\Gamma$ is a countable ICC (infinite conjugacy classes) group with property~(T), then both the outer automorphism group $\operatorname{Out}(L(\Gamma))$ and the Murray--von Neumann fundamental group $\mathcal F(L(\Gamma))$ are countable \cite{Con80}. This was one of the first manifestations of property~(T) as a rigidity phenomenon for von Neumann algebras and suggested that $L(\Gamma)$ might retain substantial information about $\Gamma$.

At that time, the main known examples of infinite discrete groups with property~(T) came from lattices in semisimple Lie groups, with higher rank lattices forming the central motivating class. For higher rank lattices, the strong rigidity theorems of Mostow and Margulis \cite{Mostow73,Mar91} show that if $\Gamma_i<G_i$, $i=1,2$, are irreducible higher rank lattices in connected semisimple real Lie groups with trivial center and no nontrivial compact factors, then an isomorphism $\Gamma_1\cong\Gamma_2$ forces an isomorphism $G_1\cong G_2$. Thus higher rank lattices provided a natural model for the idea that a sufficiently rigid analytic invariant might retain the underlying group.

Against this background, Connes proposed the following bold conjecture \cite{Co82}.

\begin{conjecture T}
Every countable ICC property~(T) group $\Gamma$ is \textbf{W$^*$-superrigid}: if $\Lambda$ is a countable discrete group such that $L(\Lambda)\cong L(\Gamma)$, then $\Lambda\cong\Gamma$ as groups.
\end{conjecture T}

An early rigidity result illustrating the information retained by group factors was obtained by Cowling--Haagerup \cite{CH89}. They proved that if $\Gamma$ is a lattice in $\operatorname{Sp}(n,1)/\{\pm1\}$, with $n\geq2$, then $n$ is determined by the group factor $L(\Gamma)$. Beginning in the early 2000s, Popa's deformation/rigidity theory produced decisive advances in the classification of group and group measure space von Neumann algebras \cite{Po06a,Po06b}. This program also yielded W$^*$-superrigidity results for group actions, notably Ioana's theorem for Bernoulli actions of property~(T) groups \cite{Io11}. At the level of group factors themselves, the first W$^*$-superrigid groups were constructed by Ioana--Popa--Vaes \cite{IPV10}. Later, Chifan--Ioana--Osin--Sun constructed the first W$^*$-superrigid groups with property~(T) \cite{CIOS21}. Subsequent work developed this circle of ideas in several directions, including rigidity of outer automorphisms and embeddings, cocycle-twisted and virtual W$^*$-superrigidity, and W$^*$-superrigidity for groups with infinite center \cite{CIOSII,DV25a,CIOSIII,CFT24,DV25b,CFOT26}.

In the decades after Connes formulated the conjecture, the class of known property~(T) groups expanded far beyond higher rank lattices, and the original conjecture came to be regarded as too broad in its full generality. At the same time, positive W$^*$-rigidity results typically relied on substantial additional algebraic, geometric, or deformation-theoretic structure. It therefore became natural to distinguish the original conjecture, formulated for arbitrary ICC property~(T) groups, from the more geometric rigidity problem for higher rank lattices. This motivates the following more focused version of Connes' conjecture, formulated by Houdayer \cite{HouICM}.

\begin{conjecture HRL}
For $i=1,2$, let $G_i$ be a connected simple real Lie group with trivial center and $\operatorname{rk}_{\mathbb R}(G_i)\geq2$, and let $\Gamma_i<G_i$ be a lattice. If $L(\Gamma_1)\cong L(\Gamma_2)$, then $G_1\cong G_2$ as Lie groups. 
\end{conjecture HRL}

This version is weaker than the restriction of Connes' original conjecture to higher rank lattices, but it retains the geometric rigidity suggested by the theorems of Mostow and Margulis. The original conjecture should therefore be viewed not merely as a prediction awaiting a yes-or-no answer, but as a guiding problem that led to increasingly precise forms of W$^*$-rigidity.

The present work starts from the possibility that property~(T), once separated from the geometric structure of higher rank lattices, may be insufficient by itself to enforce W$^*$-superrigidity. Our main result confirms this possibility by disproving Connes' original conjecture.
\begin{letterthm}\label{thm:main}
There exist explicitly defined countable discrete groups $\Gamma_1$ and $\Gamma_2$ such that:
\begin{enumerate}
  \item both $\Gamma_1$ and $\Gamma_2$ have Kazhdan's property~(T);
  \item both groups are ICC;
  \item $\Gamma_1\not\cong\Gamma_2$;
  \item there is a $*$-isomorphism
$$
    \VN(\Gamma_1)\cong\VN(\Gamma_2)
$$
  between their group von Neumann algebras.
\end{enumerate}

\end{letterthm}

\subsection*{Independent and concurrent work.}

On August 1, 2026, OpenAI publicly announced a disproof of Connes' rigidity conjecture produced by an internal version of Astra, its next major model \cite[Chapter~4]{OpenAI26}. The present construction was developed with the assistance of GPT-5.6 Sol, and a preliminary proof manuscript had been assembled by July 27, 2026. The two works were developed independently and concurrently, as is also recorded in \cite[Chapter~4, p.~110, Acknowledgments]{OpenAI26}. The OpenAI binary-carry construction produces a countably infinite family of pairwise nonisomorphic, mutually commensurable ICC property~(T) groups with isomorphic group von Neumann algebras.

Both constructions exploit the same general principle. Let $H$ be a countable discrete group, let $D_i$ be a countable discrete abelian group for $i=1,2$, and let $\theta_i:H\car D_i$ be an action by automorphisms for $i=1,2$. Denote by $\alpha_i:H\car(\widehat D_i,\mu_{\widehat D_i})$ the induced action on the Pontryagin dual equipped with its Haar probability measure. By the Fourier transform, we have
$$
L(D_i\rtimes_{\theta_i}H)
\cong
L^\infty(\widehat D_i)\rtimes_{\alpha_i}H.
$$
Consequently, the semidirect products $D_i\rtimes_{\theta_i}H$ may have isomorphic group von Neumann algebras whenever the actions $\alpha_i:H\car(\widehat D_i,\mu_{\widehat D_i})$ are conjugate as nonsingular probability-space actions, even if the conjugating map does not preserve the compact abelian group structures. This is the common mechanism underlying both counterexamples.

In the OpenAI construction, the acting group $H$ is kept fixed while the discrete abelian kernels $D_i$ are changed. The basic pair is distinguished by torsion: $D_1$ has exponent two, whereas $D_2$ contains elements of order four. By taking $H$ to be torsion-free, we ensure that every finite-order element of $D_i\rtimes_{\theta_i}H$ belongs to $D_i$. Hence $D_1\rtimes_{\theta_1}H$ has no elements of order four, whereas $D_2\rtimes_{\theta_2}H$ does, so the two groups are not isomorphic. Shifted versions of the binary-carry construction, together with an additional finite-orbit invariant, distinguish the resulting countably infinite family.

In the present construction, by contrast, both the quotient group $H$ and the elementary abelian normal subgroup $D$ are kept fixed. We change only the two actions $\theta_i:H\car D$ defining the semidirect products $\Gamma_i=D\rtimes_{\theta_i}H$. The induced actions on $\widehat D$ are conjugate as nonsingular $H$-actions through a quadratic fiber shear, but the two $H$-module structures on $D$ remain different. The nonisomorphism of $\Gamma_1$ and $\Gamma_2$ is detected by a semisimple--nonsemisimple distinction between these two $H$-module structures on $D$.

\subsection*{AI use statement.}

The construction underlying the main result of this paper was found mainly by GPT-5.6 Sol, Codex, and the Danus multi-agent research system \cite{danus}, under the author's mathematical guidance. Lean 4.32.1 was used to formally check selected parts of the argument. The resulting material was subsequently reorganized, revised, and checked by the author. All mathematical statements, proofs, references, and final editorial decisions remain the sole responsibility of the author.

\subsection*{Acknowledgements.}

The author would like to thank Professors Cyril Houdayer, Yi-Jun~Yao and Yehuda Shalom for many valuable comments on this paper. The author also thanks Soham Chakraborty for carefully reading the manuscript.

\setcounter{tocdepth}{1}
\tableofcontents

\section{Construction of the two groups}

Let $k=\Z/2\Z$ be the field with two elements. Throughout, all tensor products are taken over $k$. Set
$$
 R=k[t],
 \qquad
 A=R^3.
$$
Let $\mathrm{Flip}:A\otimes A\to A\otimes A: a\otimes b\mapsto b\otimes a$ be the flip map and define
$$
  C=(A\otimes A)^\mathrm{Flip}.
$$
Define the diagonal map
$$
 \Delta:A\longrightarrow C,
 \qquad
 \Delta(a)=a\otimes a.
$$
The group $\SL_3(R)$ acts diagonally on $A\otimes A$ and hence on $C$.

\begin{lemma}\label{lem:delta}
There is a surjective $\SL_3(R)$-equivariant linear map
$$
  \delta:C\longrightarrow A
$$
such that
$$
  \delta(\Delta(a))=a\qquad(a\in A).
$$
Moreover, the square tensors $\Delta(a)$ span $C$.
\end{lemma}

\begin{proof}
Choose a total order on the standard $k$-basis
$$
  \mathcal B=\{t^n e_j:n\geq0,\ 1\leq j\leq3\}
$$
of $A$. The fixed space $C$ has basis
$$
 \{b\otimes b:b\in\mathcal B\}
 \cup
 \{b\otimes b'+b'\otimes b:b,b'\in\mathcal B,\ b<b'\}.
$$
Define $\delta$ on this basis by
$$
 \delta(b\otimes b)=b,
 \qquad
 \delta(b\otimes b'+b'\otimes b)=0.
$$
By expanding in characteristic two, we obtain $\delta(\Delta(a))=a$. Furthermore,
$$
 \Delta(b+b')+\Delta(b)+\Delta(b')
 =b\otimes b'+b'\otimes b,
$$
so the square tensors span the displayed basis of $C$. For $l\in\SL_3(R)$,
$$
 \delta(l\Delta(a))
 =\delta(\Delta(la))
 =la
 =l\delta(\Delta(a)).
$$
Since the square tensors span $C$, this proves $\SL_3(R)$-equivariance. Surjectivity follows from $\delta(\Delta(a))=a$.
\end{proof}

Let $V=k^4$ with coordinates
$$
  v=(a_1,b_1,a_2,b_2)
$$
and standard symplectic form
$$
 \omega(v,w)=a_1b'_1+b_1a'_1+a_2b'_2+b_2a'_2.
$$
Thus, with respect to the ordered symplectic basis
$$
 e_1=(1,0,0,0),\quad f_1=(0,1,0,0),\quad
 e_2=(0,0,1,0),\quad f_2=(0,0,0,1),
$$
the matrix of $\omega$ is
$$
 J=
 \begin{pmatrix}
  0&1&0&0\\
  1&0&0&0\\
  0&0&0&1\\
  0&0&1&0
 \end{pmatrix}.
$$
The symplectic group of $(V,\omega)$ is
$$
 \Sp(V,\omega)
 =
 \{q\in\GL_k(V):
   \omega(qv,qw)=\omega(v,w)\ \text{for all }v,w\in V\}.
$$
Equivalently, in the above basis,
$$
 \Sp(V,\omega)
 =
 \{q\in\GL_4(k):q^{\mathsf T}Jq=J\}.
$$
Set
$$
 Q=\Sp(V,\omega)=\Sp_4(k)
$$
and consider the quadratic refinement
$$
 r_0(v)=a_1b_1+a_2b_2.
$$
For functionals on $V$, use the left action
$$
 (q\cdot f)(v)=f(q^{-1}v).
$$
Define
$$
 \ell_q=q\cdot r_0-r_0\qquad(q\in Q).
$$

Recall that, for a $Q$-module $M$, a map $c:Q\to M$, $q\mapsto c_q$, is a $1$-cocycle if $c_{pq}=p\cdot c_q+c_p$ for all $p,q\in Q$.
\begin{lemma}\label{lem:finite-cocycle}
Each $\ell_q$ is a linear functional in $V^*$, and the map $q\mapsto\ell_q$ is a $1$-cocycle, i.e.,
$$
  \ell_{pq}=p\cdot\ell_q+\ell_p
  \qquad(p,q\in Q).
$$
\end{lemma}

\begin{proof}
For a function $r:V\to k$ satisfying $r(0)=0$, write
$$
 B_r(v,w)=r(v+w)-r(v)-r(w)
$$
for its polarization. If $v=(a_1,b_1,a_2,b_2)$ and $w=(a'_1,b'_1,a'_2,b'_2)$, then, by direct expansion, we have
$$
\begin{aligned}
 B_{r_0}(v,w)
 &=r_0(v+w)-r_0(v)-r_0(w)\\
 &=a_1b'_1+b_1a'_1+a_2b'_2+b_2a'_2
 =\omega(v,w).
\end{aligned}
$$
For $q\in Q$, by the definition of the action and the linearity of $q^{-1}$, we have
$$
\begin{aligned}
 B_{q\cdot r_0}(v,w)
 &=r_0(q^{-1}(v+w))-r_0(q^{-1}v)-r_0(q^{-1}w)\\
 &=B_{r_0}(q^{-1}v,q^{-1}w)\\
 &=\omega(q^{-1}v,q^{-1}w)
 =\omega(v,w),
\end{aligned}
$$
where the last equality follows from $q\in\Sp(V,\omega)$. Hence
$$
 B_{\ell_q}
 =B_{q\cdot r_0}-B_{r_0}
 =\omega-\omega=0.
$$
Moreover, $\ell_q(0)=0$, so $B_{\ell_q}=0$ says precisely that
$$
 \ell_q(v+w)=\ell_q(v)+\ell_q(w)
 \qquad(v,w\in V).
$$
Thus $\ell_q$ is additive and therefore $k$-linear. Finally,
$$
 \ell_{pq}=pq\cdot r_0-r_0
 =p\cdot(q\cdot r_0-r_0)+(p\cdot r_0-r_0)
 =p\cdot\ell_q+\ell_p.
$$
\end{proof}

Let
$$
  D=(A\otimes V^*)\oplus C,
  \qquad
  H=\SL_3(R)\times Q.
$$
The action of $H$ on $A\otimes V^*$ is
$$
 (l,q)\cdot(a\otimes\varphi)=la\otimes(q\cdot\varphi).
$$
Define two actions on the elementary abelian group $D$ by
\begin{align}
 \theta_1(l,q)(u,c)
  &=\bigl((l,q)\cdot u,l\cdot c\bigr),\label{eq:theta1}\\
 \theta_2(l,q)(u,c)
  &=\bigl((l,q)\cdot u+l\delta(c)\otimes\ell_q,l\cdot c\bigr).
  \label{eq:theta2}
\end{align}

\begin{lemma}\label{lem:actions}
The maps $\theta_1$ and $\theta_2$ are actions of $H$ on $D$.
\end{lemma}

\begin{proof}
Only $\theta_2$ needs checking. Let $h=(l,p)$ and $h'=(l',q)$. Using $\delta(l'c)=l'\delta(c)$ and the fact that the $\SL_3(R)$- and $Q$-actions on $A\otimes V^*$ commute, we obtain
\begin{align*}
 \theta_2(h)\theta_2(h')(u,c)
 &=\bigl((ll',pq)\cdot u
   +ll'\delta(c)\otimes(p\cdot\ell_q+\ell_p),ll'c\bigr)\\
 &=\theta_2(ll',pq)(u,c)
\end{align*}
by Lemma~\ref{lem:finite-cocycle}.
\end{proof}

Let $D_i$ denote $D$ with the action $\theta_i$, and define
\begin{equation}\label{eq:groups}
 \boxed{\quad
 \Gamma_i=D_i\rtimes_{\theta_i}
 \bigl(\SL_3(R)\times\Sp_4(k)\bigr),
 \qquad i=1,2.
 \quad}
\end{equation}
Every constituent in \eqref{eq:groups} is countable, and $Q=\Sp_4(k)$ is finite. Thus both $\Gamma_i$ are countable discrete groups.

\section{Isomorphism of the group von Neumann algebras}\label{section: iso vna}

For any countable discrete $k$-module $M$, regard $(M,+)$ as a countable discrete abelian group and let $\widehat M=\Hom(M,\mathbb{S}^1)$ be the Pontryagin dual. For $\chi\in\widehat M$ and $m\in M$, $\chi(m)\in \mathbb{S}^1$ satisfies $\chi(m)^2=\chi(2m)=\chi(0)=1$, so $\chi(M)\subseteq\{\pm1\}$. Since the map $k=\Z/2\Z\to\{\pm1\}$ given by $s\mapsto(-1)^s$ is an isomorphism from the additive group of $k$ onto $\{\pm1\}$, every character of $M$ corresponds uniquely to a $k$-linear functional on $M$. We may therefore identify $\widehat M$ with $\operatorname{Hom}_k(M,k)$.

Under the pairing
$$
 \ip{z}{a\otimes\varphi}=\varphi(z(a)),
 \qquad
 z\in\Hom_k(A,V),
$$
the Pontryagin dual of $A\otimes V^*$ is
$$
 Z=\Hom_k(A,V).
$$
Let
$$
 \widehat C=\Hom_k(C,k),
 \qquad \widehat D=Z\times\widehat C,
$$
with their natural compact product topologies.

For
$$
 x,x'\in
 \Hom_k(A,k),
$$
regard $x\otimes x'$ as an element of $\widehat C$ by restricting the tensor-product functional on $A\otimes A$ to $C=(A\otimes A)^\mathrm{Flip}$. Write $z\in Z$ in coordinates as
$$
 z(a)=(x_1(a),x_2(a),x_3(a),x_4(a))
$$
and define
\begin{equation}\label{eq:R-def}
 \calR(z)=x_1\otimes x_2+x_3\otimes x_4\in\widehat C.
\end{equation}
On square tensors,
\begin{equation}\label{eq:R-square}
 \calR(z)(\Delta(a))=r_0(z(a)).
\end{equation}
The map $\calR:Z\to\widehat C$ is continuous because evaluation on a fixed $c\in C$ depends on only finitely many coordinates of $z$. It is also $\SL_3(R)$-equivariant: with the inverse-dual actions on $Z$ and $\widehat C$,
\begin{equation}\label{eq:SL3-equivariance-calR}
 \calR(z\circ l^{-1})(c)
 =\calR(z)((l^{-1}\otimes l^{-1})c).
\end{equation}
It suffices to verify this on the spanning square tensors, where it follows from \eqref{eq:R-square}.

For $h\in H$, let $\alpha_{i,h}$ be the inverse dual transformation
$$
 \alpha_{i,h}(\chi)=\chi\circ\theta_i(h)^{-1}
 \qquad(\chi\in\widehat D).
$$
For $l\in\SL_3(R)$ and $(z,y)\in \widehat{D}=Z\times \widehat C$, the two transformations coincide:
$$
 \alpha_{i,(l,1)}(z,y)
 =\bigl(z\circ l^{-1},y\circ(l^{-1}\otimes l^{-1})\bigr).
$$
For $q\in Q$,
$$
 \alpha_{1,(1,q)}(z,y)=(qz,y),
$$
where $(qz)(a)=q(z(a))$.

Since $\theta_2$ is a group action,
$$
\theta_2(1,q)^{-1}
=
\theta_2(1,q^{-1}),
$$
and therefore for $(u,c)\in D=(A\otimes V^*)\oplus C$
$$
\theta_2(1,q)^{-1}(u,c)
=
\bigl(q^{-1}\cdot u+\delta(c)\otimes\ell_{q^{-1}},c\bigr).
$$
For $(z,y)\in \widehat{D}=Z\times\widehat C$ and $(u,c)\in D$, the dual pairing is
$$
(z,y)(u,c)=\ip{z}{u}+y(c).
$$
It follows that
$$
\begin{aligned}
\bigl(\alpha_{2,(1,q)}(z,y)\bigr)(u,c)
&=(z,y)\bigl(\theta_2(1,q)^{-1}(u,c)\bigr)\\
&=\ip{z}{q^{-1}\cdot u}
  +\ip{z}{\delta(c)\otimes\ell_{q^{-1}}}
  +y(c)\\
&=\ip{qz}{u}
  +\ell_{q^{-1}}\bigl(z(\delta(c))\bigr)
  +y(c).
\end{aligned}
$$
Define $J_q(z)\in\widehat C=\operatorname{Hom}_k(C,k)$ by
\begin{equation}\label{eq:J-def}
J_q(z)(c)
=
\ell_{q^{-1}}\bigl(z(\delta(c))\bigr).
\end{equation}
Then the preceding computation becomes
$$
\bigl(\alpha_{2,(1,q)}(z,y)\bigr)(u,c)
=
\ip{qz}{u}+\bigl(y+J_q(z)\bigr)(c),
$$
and hence
\begin{equation}\label{eq:dual-triangular}
\alpha_{2,(1,q)}(z,y)
=
\bigl(qz,y+J_q(z)\bigr).
\end{equation}

\begin{lemma}\label{lem:defect}
For every $q\in Q$ and $z\in Z$,
$$
 J_q(z)=\calR(qz)-\calR(z).
$$
\end{lemma}

\begin{proof}
By equation~(\ref{eq:R-square}), for every $a\in A$, we have
\begin{align*}
 J_q(z)(\Delta(a))
 &=\ell_{q^{-1}}(z(a))\\
 &=r_0(qz(a))-r_0(z(a))\\
 &=\calR(qz)(\Delta(a))
   -\calR(z)(\Delta(a)).
\end{align*}
Since $\Delta(A)$ spans $C$ by Lemma~\ref{lem:delta}, the two functionals agree on all of $C$.
\end{proof}

Define the fiber shear
\begin{equation}\label{eq:F-def}
 F:\widehat D\longrightarrow\widehat D,
 \qquad F(z,y)=(z,y+\calR(z)).
\end{equation}

\begin{proposition}\label{prop:dual-conjugacy}
The map $F$ is a Haar-measure-preserving homeomorphism and
$$
 F\alpha_{1,h}F^{-1}=\alpha_{2,h}\qquad(h\in H).
$$
\end{proposition}

\begin{proof}
The inverse of $F$ is $F^{-1}(z,y)=(z,y-\calR(z))$. Since $\widehat C$ is a $\Z/2\Z$-module, one also has $F^{-1}=F$. Thus continuity of $\calR$ makes $F$ a homeomorphism. For $q\in Q$, by Lemma~\ref{lem:defect} and \eqref{eq:dual-triangular}, we have
\begin{align*}
 F\alpha_{1,(1,q)}F^{-1}(z,y)
 &=F\alpha_{1,(1,q)}\bigl(z,y-\calR(z)\bigr)\\
 &=F\bigl(qz,y-\calR(z)\bigr)\\
 &=\bigl(qz,y-\calR(z)+\calR(qz)\bigr)\\
 &=\bigl(qz,y+J_q(z)\bigr)\\
 &=\alpha_{2,(1,q)}(z,y).
\end{align*}

For $l\in \SL_3(R)$,
\begin{align*}
F\alpha_{1,(l,1)}F^{-1}(z,y)
&=
F\alpha_{1,(l,1)}\bigl(z,y-\calR(z)\bigr)\\
&=
F\Bigl(
z\circ l^{-1},
\bigl(y-\calR(z)\bigr)\circ
(l^{-1}\otimes l^{-1})
\Bigr)\\
&=
\Bigl(
z\circ l^{-1},
y\circ(l^{-1}\otimes l^{-1})
-\calR(z)\circ(l^{-1}\otimes l^{-1})
+\calR(z\circ l^{-1})
\Bigr)\\
&=
\Bigl(
z\circ l^{-1},
y\circ(l^{-1}\otimes l^{-1})
\Bigr)\\
&=
\alpha_{2,(l,1)}(z,y),
\end{align*}
where the fourth equality follows from \eqref{eq:SL3-equivariance-calR}.

Since $\SL_3(R)$ and $Q$ generate $H$, strict conjugacy follows for every $h\in H$.

Let $\mu_Z$ and $\mu_{\widehat C}$ be the Haar probability measures and define $\mu_{\widehat D}=\mu_Z\otimes\mu_{\widehat C}$. For every bounded Borel function $g$,
\begin{align*}
 \int_{\widehat D} g\circ F\,{\d}\mu_{\widehat D}
 &=\int_Z\int_{\widehat C}
   g(z,y+\calR(z))\,{\d}\mu_{\widehat C}(y)\,{\d}\mu_Z(z)\\
 &=\int_Z\int_{\widehat C}
   g(z,y)\,{\d}\mu_{\widehat C}(y)\,{\d}\mu_Z(z),
\end{align*}
because translation by $\calR(z)$ preserves Haar measure on each $\widehat C$-fiber.
\end{proof}

\begin{remark}\label{rem:not-additive}
The map $\calR$ is generally quadratic rather than additive, so $F$ need not be a compact-group automorphism. The proof uses only that $F$ is a Haar-preserving homeomorphism and a strict conjugacy of the two actions.
\end{remark}

\begin{proposition}\label{prop:factor-iso}
There is a $*$-isomorphism
$$
\VN(\Gamma_1)\cong\VN(\Gamma_2).
$$
\end{proposition}

\begin{proof}
Let $D_i$ denote $D$ with the action $\theta_i$. Then for $i=1,2$,
$$\Gamma_i=D_i\rtimes H.$$
By Fourier transform on the discrete abelian group $D_i$, we have
$$
 \VN(\Gamma_i)\cong L(D_i)\rtimes H
 \cong L^\infty(\widehat D_i,\mu_{\widehat D_i})
 \rtimes H.
$$
Define $F:(\widehat D_1,\mu_{\widehat D_1})\to (\widehat D_2,\mu_{\widehat D_2})$ as in (\ref{eq:F-def}). By Proposition~\ref{prop:dual-conjugacy}, $F$ is an $H$-equivariant measure-preserving isomorphism between the two nonsingular $H$-spaces. Therefore,
$$\VN(\Gamma_1)\cong L^\infty(\widehat D_1,\mu_{\widehat D_1})\rtimes H\cong L^\infty(\widehat D_2,\mu_{\widehat D_2})\rtimes H\cong L(\Gamma_2).$$
\end{proof}

\section{Property (T)}
We first prove that $\SL_3(R)$ has property~(T). The systematic study of property~(T) for linear groups over general rings was initiated by Shalom's work on bounded generation and relative property~(T) \cite{Sha99}. Shalom subsequently proved that $\EL_n(S)$ has property~(T) for every finitely generated commutative unital ring $S$ satisfying $n\geq 2+\operatorname{Kdim}(S)$ \cite[Theorem~1.1]{Sha06}, where $\operatorname{Kdim}(S)$ denotes the Krull dimension of $S$. This applies directly to $R=(\Z/2\Z)[t]$. Subsequent work established broader versions for arbitrary finitely generated associative rings and for groups graded by root systems \cite{EJ,EJK}.

Let $\EL_3(R)$ denote the subgroup of $\GL_3(R)$ generated by the elementary matrices $I+rE_{ij}$, where $i,j\in\{1,2,3\}$, $i\ne j$, and $r\in R$.
\begin{proposition}\label{prop:SL3-T}
For $R=k[t]$, one has
$$
 \SL_3(R)=\EL_3(R),
$$
and this group has property~(T).
\end{proposition}

\begin{proof}
We first prove that $\SL_3(R)=\EL_3(R)$. A much more general result $$\SL_n(F[t_1,t_2,...,t_m])=\EL_n(F[t_1,t_2,...,t_m])$$ for general field $F$ and $n\geq3$ was established by Suslin \cite[Corollary~6.7]{Sus77}; for completeness, we include an elementary proof for the present special case $R=(\Z/2\Z)[t]$.

For $i,j\in\{1,2,3\}$ with $i\ne j$ and $r\in R$, let
$$
 e_{ij}(r)=I+rE_{ij}.
$$
By definition, $\EL_3(R)$ is generated by these elementary matrices. Since every $e_{ij}(r)$ has determinant one,
$$
 \EL_3(R)\leq\SL_3(R).
$$

Conversely, $R=k[t]$ is a Euclidean domain with degree as a Euclidean function, and its only unit is $1$. Before applying the Euclidean algorithm, observe that both row additions and row interchanges can be performed using elementary matrices. Left multiplication by $I+fE_{ij}$ performs the row operation
$$
\operatorname{row}_i
\longmapsto
\operatorname{row}_i+f\operatorname{row}_j.
$$
Moreover, in characteristic two,
$$
\begin{pmatrix}
0&1\\
1&0
\end{pmatrix}
=
\begin{pmatrix}
1&1\\
0&1
\end{pmatrix}
\begin{pmatrix}
1&0\\
1&1
\end{pmatrix}
\begin{pmatrix}
1&1\\
0&1
\end{pmatrix}.
$$
Embedding this identity into any two coordinates of a $3\times3$ matrix shows that every row interchange is also a product of elementary matrices.

Now let $g\in\SL_3(R)$ and write its first column as $(a,b,c)^{\mathsf T}$. Expanding $\det(g)=1$ along the first column produces $x,y,z\in R$ such that
$$
ax+by+cz=1.
$$
Thus
$$
(a,b,c)=R.
$$
The Euclidean algorithm, implemented by row additions and row interchanges, therefore transforms the first column into $(1,0,0)^{\mathsf T}$. Hence there exists a product $u_1$ of elementary matrices such that
$$
u_1g=
\begin{pmatrix}
1&x_1&x_2\\
0&a_1&a_2\\
0&a_3&a_4
\end{pmatrix}.
$$
Since every elementary matrix has determinant one,
$$
1=\det(u_1g)
=\det
\begin{pmatrix}
a_1&a_2\\
a_3&a_4
\end{pmatrix}.
$$
Thus the lower-right $2\times2$ block belongs to $\SL_2(R)$. Applying the Euclidean algorithm to its first column, using only operations on the second and third rows, reduces that column to $(1,0)^{\mathsf T}$. The determinant condition then forces the lower-right entry to be $1$, and one final row addition clears its remaining off-diagonal entry. Consequently, another product $u_2$ of elementary matrices gives
$$
u_2u_1g=
\begin{pmatrix}
1&x_1&x_2\\
0&1&0\\
0&0&1
\end{pmatrix}.
$$
Adding $x_1$ times the second row and $x_2$ times the third row to the first row reduces this matrix to $I$. We have therefore found $u_1,u_2,u_3\in \EL_3(R)$ such that
$$
 u_3u_2 u_1g=I.
$$
The inverse of an elementary matrix is elementary, so
$$
 g=u_1^{-1}u_2^{-1} u_3^{-1}\in\EL_3(R).
$$
Therefore
$$
\SL_3(R)=\EL_3(R).
$$

The ring $R=(\Z/2\Z)[t]$ is a finitely generated commutative unital ring of Krull dimension one. Hence, by \cite[Theorem~1.1]{Sha06} (see also \cite{EJ,EJK}), the group $\EL_3(R)$ has property~(T). The equality just proved therefore implies that $\SL_3(R)$ has property~(T).
\end{proof}

Recall that a pair $(G,N)$ of groups with $N<G$ has \textbf{relative property~(T)} if every unitary representation of $G$ admitting almost invariant unit vectors admits a nonzero $N$-invariant vector. We now prove that $(D_i\rtimes\SL_3(R),D_i)$ has relative property~(T). The proof strategy is in the same spirit as Shalom's argument in \cite{Sha99}, which establishes relative property~(T) for $(S^2\rtimes \SL_2(S),S^2)$, where $S$ is a finitely generated commutative unital ring.

For $a,b\in A$, let
$$
 \beta(a,b)=a\otimes b+b\otimes a\in C=(A\otimes A)^\mathrm{Flip}.
$$
Then
$$
 \Delta(a+b)
 =\Delta(a)+\Delta(b)+\beta(a,b).
$$
For $i=1,2,3$, let $e_i\in A=R^3$ be
$$e_1=(1,0,0),\, e_2=(0,1,0),\, e_3=(0,0,1).$$
\begin{lemma}
\label{lem:chart-exhaustion}
For $N\geq1$, let
$$
 \mathcal P_N
 =\operatorname{span}_k\{1,t,\ldots,t^{N-1}\}
 \subset R=k[t].
$$
For $\{s,j,m\}=\{1,2,3\}$, set
$$
 U_{s,N}=\{e_s+fe_j+he_m:f,h\in\mathcal P_N\}
$$
and
$$
 C_N=\operatorname{span}_k
 \{\Delta(v):v\in U_{1,N}\cup U_{2,N}\cup U_{3,N}\}.
$$
Then $C_N\subset C_{N+1}$ and $\bigcup_{N\geq1}C_N=C$.
\end{lemma}

\begin{proof}
The space $C$ has the basis displayed in the proof of Lemma~\ref{lem:delta}. We obtain its members from chart squares. For pairwise distinct $i,j,m$ and $w=fe_j$, by expanding in characteristic two, we obtain
\begin{align}
 \Delta(w)={}&\Delta(e_i+w)+\Delta(e_i)
 +\Delta(e_m+w)+\Delta(e_m)\notag\\
 &+\Delta(e_i+e_m+w)+\Delta(e_i+e_m).
 \label{eq:six-square}
\end{align}
Every square on the right lies in a chart for sufficiently large $N$. For pairwise distinct $i,j,m$,
\begin{align}
 \beta(fe_i,he_j)={}&\Delta(e_m+fe_i+he_j)
 +\Delta(e_m+fe_i)\notag\\
 &+\Delta(e_m+he_j)+\Delta(e_m),
 \label{eq:cross-coordinate}
\end{align}
while
\begin{equation}
 \beta(fe_i,he_i)
 =\Delta((f+h)e_i)+\Delta(fe_i)+\Delta(he_i).
 \label{eq:same-coordinate}
\end{equation}
The squares in \eqref{eq:same-coordinate} lie in the union by \eqref{eq:six-square}. Taking $f$ and $h$ to be monomials produces every diagonal and symmetric cross member of the basis of $C$.
\end{proof}
Recall that a \textbf{Boolean polynomial} in $n$ variables over $k=\Z/2\Z$ is a polynomial function on $k^n$, or equivalently an element of $k[x_1,\ldots,x_n]/(x_1^2-x_1,\ldots,x_n^2-x_n)$. Every function $k^n\to k$ admits a unique square-free representative, called its algebraic normal form, and its degree is the largest number of distinct variables occurring in a monomial with nonzero coefficient. For example, $r_0(a_1,b_1,a_2,b_2)=a_1b_1+a_2b_2$ is a Boolean polynomial of degree two. The \textbf{Hamming weight} of a Boolean polynomial $P$ is the cardinality of its support, namely $$\operatorname{wt}(P):=|\operatorname{supp}(P)|=|\{x\in k^n\mid P(x)=1\}|.$$ For $v\in k^n$, the directional derivative of $P$ in the direction $v$ is the Boolean polynomial $\Delta_vP$ defined by $$(\Delta_vP)(x):=P(x+v)-P(x).$$ If $P$ has degree at most $r$, then $\Delta_vP$ has degree at most $r-1$.
\begin{lemma}\label{lem:boolean-weight}
Let $P:k^n\to k$ be a nonzero Boolean polynomial of algebraic degree at most two. Then
$$
 \operatorname{wt}(P)=|\{x\in k^n:P(x)=1\}|\geq2^{n-2}.
$$
\end{lemma}

\begin{proof}
If $P$ is affine, it is either the constant one function or has weight $2^{n-1}$. Otherwise, $P$ is of algebraic degree 2. Choose $i\ne j$ such that the coefficient of $x_ix_j$ in the unique square-free representative of $P$ is one. Since
$$
\Delta_{e_i}(x_ix_j)
=(x_i+1)x_j-x_ix_j
=x_j,
$$
the coefficient of $x_j$ in $\Delta_{e_i}P$ is one. Hence
$$
\Delta_{e_i}P(x)=P(x+e_i)-P(x)
$$
is a nonconstant affine function, and therefore
$$
\bigl|\{x\in k^n\mid \Delta_{e_i}P(x)=1\}\bigr|=2^{n-1}.
$$
Moreover,
$$
\begin{aligned}
\Delta_{e_i}P(x+e_i)
&=P(x+2e_i)-P(x+e_i)\\
&=P(x)-P(x+e_i)\\
&=P(x+e_i)-P(x)\\
&=\Delta_{e_i}P(x).
\end{aligned}
$$
Here we used $\operatorname{char}(k)=2$. Thus $\Delta_{e_i}P$ is constant on every pair $\{x,x+e_i\}$. Since each pair has two elements, exactly $ \frac{2^{n-1}}{2}=2^{n-2} $ of these pairs satisfy $\Delta_{e_i}P=1$. On each such pair,
$$
P(x+e_i)-P(x)=1,
$$
and hence
$$
\{P(x),P(x+e_i)\}=\{0,1\}.
$$
Therefore exactly one endpoint of each of these $2^{n-2}$ pairs belongs to $\operatorname{supp}(P)$, so
$$
\operatorname{wt}(P)
=\sum_{\{x,x+e_i\}}
\bigl|\{x,x+e_i\}\cap\operatorname{supp}(P)\bigr|
\geq 2^{n-2}.
$$
\end{proof}

As in Section~\ref{section: iso vna}, for every countable discrete $k$-module $M$ we continue to identify its Pontryagin dual $\widehat M$ with $\Hom_k(M,k)$. Under this identification, the trivial character is the zero functional.

\begin{proposition}\label{prop:tensor-detector}
If $\mu$ is an $\SL_3(R)$-invariant Borel probability measure on $\widehat C$, then
\begin{equation}
 \mu(\widehat C\setminus\{0\})
 \leq12\,\mu\{\chi\in\widehat C\mid\chi(\Delta(e_1))=1\}.
 \label{eq:C-detector}
\end{equation}
If $\nu$ is an $\SL_3(R)$-invariant Borel probability measure on $\widehat A$, then
\begin{equation}
 \nu(\widehat A\setminus\{0\})
 \leq12\,\nu\{\psi\in\widehat A\mid\psi(e_1)=1\}.
 \label{eq:A-detector}
\end{equation}
\end{proposition}

\begin{proof}
Let $\CP_N$ and $U_{s,N}$ be as in Lemma~\ref{lem:chart-exhaustion}. Then through the map
$$f=a_{N-1}t^{N-1}+...+a_0\in \CP_N\mapsto (a_0,a_1,...,a_{N-1})\in k^N,$$
we have $\CP_N\cong k^N$ as $k$-modules. We also have $U_{s,N}\cong \CP_N\oplus\CP_N\cong k^{2N}$.

For $\chi\in\widehat C=\Hom_k(C,k)$, the function
$$
 (f,h)\longmapsto
 \chi\bigl(\Delta(e_s+fe_j+he_m)\bigr)
$$
on the chart $U_{s,N}\cong k^{2N}$ is a Boolean polynomial of degree at most two in the $2N$ coefficients of $f$ and $h$. Define
$$
 B_N(\chi)=\frac{1}{3\cdot2^{2N}}
 \sum_{s=1}^3\left(\sum_{v\in U_{s,N}}
 \mathbf{1}_{\{\rho\in\widehat C\mid\rho(\Delta(v))=1\}}(\chi)\right).
$$
If $\chi$ does not annihilate $C_N=\operatorname{span}_k \{\Delta(U_{1,N}\cup U_{2,N}\cup U_{3,N})\}$, at least one of the three chart polynomials is nonzero. By Lemma~\ref{lem:boolean-weight}, $B_N(\chi)\geq1/12$.

Every $v\in U_{s,N}$ belongs to the $\SL_3(R)$-orbit of $e_1$: an even coordinate permutation first sends $e_1$ to $e_s$, and the elementary matrices
$$
 I+fE_{js},\qquad I+hE_{ms}
$$
then produce $e_s+fe_j+he_m$. By the identity $\Delta(lv)=l\Delta(v)$ and the invariance of $\mu$, we have
$$
 \int_{\widehat C}B_N(\chi)\,{\d}\mu(\chi)
 =\mu\{\chi\in\widehat C\mid\chi(\Delta(e_1))=1\}.
$$
Therefore,
$$
 \mu\{\chi\in\widehat C\mid\chi(\Delta(e_1))=1\}
 \geq\frac1{12}\mu(\widehat C\setminus\Ann(C_N)),
$$
where $\Ann(C_N)=\{\chi\in \widehat C\mid \chi|_{C_N}=0\}$. By Lemma~\ref{lem:chart-exhaustion}, $\Ann(C_N)\searrow\{0\}$. Therefore, continuity from above proves \eqref{eq:C-detector}.

Dualizing the $\SL_3(R)$-equivariant surjection
$$
 \delta:C\longrightarrow A
$$
gives the injective equivariant map
$$
 \widehat\delta:\widehat A\longrightarrow\widehat C,
 \qquad
 \widehat\delta(\psi)=\psi\circ\delta.
$$
It preserves nonzero elements and satisfies $\widehat\delta(\psi)(\Delta(e_1))=\psi(e_1)$. Applying \eqref{eq:C-detector} to $(\widehat\delta)_*\nu$ proves \eqref{eq:A-detector}.
\end{proof}

Choose a basis $\varphi_1,\ldots,\varphi_4$ of $V^*=(k^4)^*$. Restricted to $\SL_3(R)$, both $D_i$ have the same module structure
$$
 D_i=(A\otimes V^*)\oplus C\cong A^{\oplus4}\oplus C.
$$
Define detector elements
$$
 d_r=(e_1\otimes\varphi_r,0)\quad(1\leq r\leq4),
 \qquad
 d_5=(0,\Delta(e_1)).
$$

\begin{corollary}\label{cor:five-detector}
For every $\SL_3(R)$-invariant Borel probability measure $\mu$ on $\widehat{D_i}$,
\begin{equation}
 \mu(\widehat{D_i}\setminus\{0\})
 \leq12\sum_{r=1}^5
 \mu\{\chi\in\widehat{D_i}\mid\chi(d_r)=1\}.
 \label{eq:five-detector}
\end{equation}
\end{corollary}

\begin{proof}
Under
$$
\widehat{D_i}\cong(\widehat A)^4\times\widehat C,
$$
write
$$
\chi=(\psi_1,\ldots,\psi_4,\eta)\in \widehat{D_i}.
$$
For $1\leq r\leq4$, let $\mu_r$ be the pushforward of $\mu$ under the projection $\chi\mapsto\psi_r$, and let $\mu_5$ be its pushforward under the projection $\chi\mapsto\eta$. Since these coordinate projections are $\SL_3(R)$-equivariant, every $\mu_r$ is $\SL_3(R)$-invariant. Set
$$
m_r=\mu_r(\widehat A\setminus\{0\})\quad(1\leq r\leq4),
\qquad
m_5=\mu_5(\widehat C\setminus\{0\}),
$$
and
$$
p_r=\mu\{\chi\in\widehat{D_i}\mid\chi(d_r)=1\}
\quad(1\leq r\leq5).
$$
By applying Proposition~\ref{prop:tensor-detector} to the corresponding marginal, we have
$$
m_r\leq12p_r
\quad(1\leq r\leq5).
$$
Moreover, $\chi\neq0$ if and only if at least one of its five coordinates is nonzero. Hence the union bound gives
$$
\mu(\widehat{D_i}\setminus\{0\})
\leq\sum_{r=1}^5m_r
\leq12\sum_{r=1}^5p_r,
$$
which is \eqref{eq:five-detector}. No independence assumption is used.
\end{proof}

\begin{proposition}\label{prop:relative-T}
For $i=1,2$, the pair
$$
 \left(
  D_i\rtimes\SL_3(R),D_i
 \right)
$$
has relative property~(T).
\end{proposition}

\begin{proof}
Assume that a unitary representation $\pi:G_i=D_i\rtimes\SL_3(R)\to \mathcal{U}(K)$ has almost invariant unit vectors $\xi_n$. It remains to prove that the $D_i$-fixed subspace of $K$ is nonzero. Let $P_{\mathrm{fix}}$ be the orthogonal projection onto the $\SL_3(R)$-fixed subspace. Property~(T) of $\SL_3(R)$ implies
$$
 \norm{(1-P_{\mathrm{fix}})\xi_n}\longrightarrow0.
$$
Indeed, otherwise, after passing to a subsequence, there would exist $\delta_0>0$ such that
$$
\norm{(1-P_{\mathrm{fix}})\xi_n}\geq\delta_0
$$
for every $n$. Define
$$
\zeta_n=
\frac{(1-P_{\mathrm{fix}})\xi_n}
     {\norm{(1-P_{\mathrm{fix}})\xi_n}}
\in\operatorname{Ran}(1-P_{\mathrm{fix}}).
$$
Since $P_{\mathrm{fix}}$ commutes with $\pi(h)$ for every $h\in\SL_3(R)$, for every finite subset $F\subseteq\SL_3(R)$ we have
$$
\max_{h\in F}\norm{\pi(h)\zeta_n-\zeta_n}
\leq
\frac{1}{\delta_0}
\max_{h\in F}\norm{\pi(h)\xi_n-\xi_n}
\longrightarrow0.
$$
Thus $(\zeta_n)$ is a sequence of almost invariant unit vectors for the restriction of $\pi$ to $\operatorname{Ran}(1-P_{\mathrm{fix}})$, whereas
$$
\operatorname{Ran}(1-P_{\mathrm{fix}})^{\SL_3(R)}=\{0\}.
$$
This contradicts property~(T) of $\SL_3(R)$. Hence
$$
 \eta_n
 =\frac{P_{\mathrm{fix}}\xi_n}{\norm{P_{\mathrm{fix}}\xi_n}}
$$
is eventually defined, is $\SL_3(R)$-fixed, and satisfies $\norm{\eta_n-\xi_n}\to0$. For each fixed $d\in D_i$,
$$
 \norm{\pi(d)\eta_n-\eta_n}
 \leq2\norm{\eta_n-\xi_n}
 +\norm{\pi(d)\xi_n-\xi_n}\longrightarrow0.
$$

Since $G_i$ is countable, the closed invariant subspace generated by
$$
 \{\pi(g)\eta_n:g\in G_i,\ n\geq1\}
$$
is separable; restrict the unitary representation to it if necessary. Since $D_i$ is abelian and we have identified $\widehat{D_i}$ with $\mathrm{Hom}_k(D_i,k)$, the Stone--Naimark--Ambrose--Godement (SNAG) theorem \cite[Theorem~D.3.1]{BHV08} provides a projection-valued Borel measure $E$ on $\widehat{D_i}=\mathrm{Hom}_k(D_i,k)$ such that
$$
  \pi(d)
  =
  \int_{\widehat{D_i}} (-1)^{\chi(d)}\,{\d}E(\chi),
  \qquad d\in D_i.
$$
Thus, for every Borel set $B\subseteq \widehat{D_i}$, the operator $E(B)$ is the orthogonal projection onto the spectral subspace corresponding to $B$. In particular, $E(\{0\})$ is the orthogonal projection onto the $D_i$-fixed subspace.

Define a Borel probability measure $\mu_n$ on $\widehat{D_i}$ by
$$
\mu_n(B)=\langle E(B)\eta_n,\eta_n\rangle
$$
for every Borel set $B\subseteq\widehat{D_i}$. The covariance relation
$$
E(h\cdot B)=\pi(h)E(B)\pi(h)^*
$$
and the $\SL_3(R)$-invariance of $\eta_n$ imply that
$$
\begin{aligned}
\mu_n(h\cdot B)
&=\langle E(h\cdot B)\eta_n,\eta_n\rangle\\
&=\langle\pi(h)E(B)\pi(h)^*\eta_n,\eta_n\rangle\\
&=\langle E(B)\eta_n,\eta_n\rangle
=\mu_n(B)
\end{aligned}
$$
for every $h\in\SL_3(R)$. Thus $\mu_n$ is $\SL_3(R)$-invariant.

Each $d_r$ has order two, and therefore

$$
\begin{aligned}
\norm{\pi(d_r)\eta_n-\eta_n}^2
&=
\int_{\widehat{D_i}}
\left|(-1)^{\chi(d_r)}-1\right|^2\,{\d}\mu_n(\chi)\\
&=
4\mu_n\{\chi\in\widehat{D_i}\mid\chi(d_r)=1\}.
\end{aligned}
$$
Set
$$
p_{r,n}
=
\mu_n\{\chi\in\widehat{D_i}\mid\chi(d_r)=1\}.
$$
Since $(\eta_n)$ is almost invariant under $G_i$, we have
$$
p_{r,n}
=
\frac14\norm{\pi(d_r)\eta_n-\eta_n}^2
\longrightarrow0
$$
for every $1\leq r\leq5$. On the other hand, by Corollary~\ref{cor:five-detector}, we have
$$
\mu_n(\widehat{D_i}\setminus\{0\})
\leq
12\sum_{r=1}^5p_{r,n}\to 0.
$$
Hence
$$\langle E(\{0\})\eta_n,\eta_n\rangle=1- \mu_n(\widehat{D_i}\setminus\{0\})\to 1.$$
In particular, $E(\{0\})\not=0$ and the $D_i$-fixed subspace is nonzero.
\end{proof}

The following lemma is well known; see, for example, \cite[Proposition~1.3]{BR95}.
\begin{lemma}\label{lem:extension}
Let $G$ be a countable discrete group and let $N\triangleleft G$. If $(G,N)$ has relative property~(T) and $G/N$ has property~(T), then $G$ has property~(T).
\end{lemma}

\begin{proposition}\label{prop:Gamma-T}
Both $\Gamma_1$ and $\Gamma_2$ have property~(T).
\end{proposition}

\begin{proof}
Apply Lemma~\ref{lem:extension} to
$$
 D_i\triangleleft D_i\rtimes\SL_3(R).
$$
Proposition~\ref{prop:relative-T} supplies relative property~(T), and the quotient is $\SL_3(R)$, which has property~(T) by Proposition~\ref{prop:SL3-T}. Hence $D_i\rtimes\SL_3(R)$ has property~(T). It is normal of finite index in $\Gamma_i$, and
$$
 \Gamma_i/\bigl(D_i\rtimes\SL_3(R)\bigr)\cong Q.
$$
Property~(T) of $D_i\rtimes\SL_3(R)$ implies relative property~(T) for $
 \bigl(\Gamma_i,D_i\rtimes\SL_3(R)\bigr),
$ because a $\Gamma_i$-almost-invariant sequence is almost invariant for the subgroup. The finite group $Q$ has property~(T), so a second application of Lemma~\ref{lem:extension} proves the result.
\end{proof}

\section{The ICC property}

\begin{lemma}\label{lem:linear-icc}
The group $\SL_3(R)$ is ICC.
\end{lemma}

\begin{proof}
Suppose $g\in\SL_3(R)$ has finite conjugacy class. Then the centralizer $C_{\SL_3(R)}(g)$ has finite index. Fix $r,s\in\{1,2,3\}$ with $r\ne s$, define
$$
 U_{rs}=\{I+fE_{rs}:f\in R\},
$$
which is a subgroup of $\SL_3(R)$. Since $C_{\SL_3(R)}(g)$ has finite index in $\SL_3(R)$, we have
$$
[U_{rs}:U_{rs}\cap C_{\SL_3(R)}(g)]
\leq
[\SL_3(R):C_{\SL_3(R)}(g)]
<\infty.
$$
Since $U_{rs}\cong(R,+)$ is infinite, the intersection $U_{rs}\cap C_{\SL_3(R)}(g)$ is nontrivial. Hence it contains $I+fE_{rs}$ for some nonzero $f\in R$. By commutation and the domain property, $gE_{rs}=E_{rs}g$. Thus $g$ commutes with every off-diagonal matrix unit and is scalar. The only unit of $R$ is $1$, so $g=I$.
\end{proof}

\begin{lemma}\label{lem:D-infinite-orbits}
Every nonzero element of $D_i$ has an infinite $\SL_3(R)$-orbit.
\end{lemma}

\begin{proof}
First suppose $a\in A=R^3$ has finite orbit. With the same argument as in the proof of Lemma~\ref{lem:linear-icc}, its stabilizer has finite index, and contains $I+fE_{rs}$ for some nonzero $f$ for each $r,s\in\{1,2,3\}$ with $r\ne s$. The identity $(I+fE_{rs})a=a$ gives $fa_s=0$, hence $a_s=0$. Varying $s$ gives $a=0$.

For $C=(A\otimes A)^\mathrm{Flip}$, identify
$$
 A\otimes A=k[t]^3\otimes k[t]^3
 \cong \mathrm{M}_3(k[t_1,t_2])
$$
by
$$
 (f(t)e_r)\otimes(g(t)e_s)
 \longmapsto f(t_1)g(t_2)E_{rs}.
$$
The diagonal action of $u_f=I+fE_{rs}$ is
$$
 P\longmapsto
 (I+f(t_1)E_{rs})P(I+f(t_2)E_{sr}).
$$
If $P=(P_{rs})_{3\times 3}\in \mathrm{M}_3(k[t_1,t_2])$ has finite orbit, with the same argument above, its stabilizer meets $U_{rs}$ in a finite-index additive subgroup and therefore contains parameters $f$ of arbitrarily high
degree. Hence we can choose $f$ with $u_f\cdot P=P$ and $$d=\deg(f)>\max \{\deg(P_{rs})\mid 1\leq r,s \leq 3\},$$
where $\deg(P_{rs})$ denotes the total $(t_1,t_2)$-degree of $P_{rs}$. By the invariance $u_f \cdot P=P$, we have
\begin{equation}
 f(t_1)E_{rs}P+f(t_2)PE_{sr}
 +f(t_1)f(t_2)E_{rs}PE_{sr}=0.
 \label{eq:C-orbit}
\end{equation}
The $(r,r)$-entry of \eqref{eq:C-orbit} is
\begin{equation}\label{eq:C-orbit-rr}
 f(t_1)P_{sr}+f(t_2)P_{rs}
 +f(t_1)f(t_2)P_{ss}=0.
\end{equation}
If $P_{ss}\ne0$, then by our choice of $f$, the total $(t_1,t_2)$-degrees satisfy
$$\deg(f(t_1)f(t_2)P_{ss})\geq 2d>\deg(f(t_1)P_{sr}+f(t_2)P_{rs}).$$
This contradicts \eqref{eq:C-orbit-rr}. Hence $P_{ss}=0$. Since
$$
 E_{rs}PE_{sr}=P_{ss}E_{rr},
$$
the term $f(t_1)f(t_2)E_{rs}PE_{sr}$ in \eqref{eq:C-orbit} vanishes, and that equation reduces to
$$
 f(t_1)E_{rs}P+f(t_2)PE_{sr}=0.
$$
If $E_{rs}P\ne0$, by our choice of $f$, the first term has $t_1$-degree strictly larger than the second term, which is impossible. Thus $E_{rs}P=0$, and then $PE_{sr}=0$. Hence row $s$ and column $s$ vanish. By varying $s$, we obtain $P=0$. Therefore every nonzero element of $C$ has infinite orbit.

Finally, since $\SL_3(R)$ acts trivially on $V^*$, its action on $A\otimes V^*$ is the direct sum of four copies of its action on $A$. The $\SL_3(R)$-action on $D_i=(A\otimes V^*)\oplus C$ is independent of $i$. Every nonzero element of $D_i$ has a nonzero component in either one of the four copies of $A$ or in $C$. Since every such nonzero component has an infinite orbit, so does the original element.
\end{proof}

\begin{proposition}\label{prop:Gamma-icc}
Both $\Gamma_1$ and $\Gamma_2$ are ICC.
\end{proposition}

\begin{proof}
Let
$$
 \gamma=(d;(l,q))\in D_i\rtimes(\SL_3(R)\times Q)
$$
be nontrivial. If $l\ne I$, since $\SL_3(R)$ is ICC by Lemma~\ref{lem:linear-icc}, the projection of its conjugacy class to $\SL_3(R)$ contains the infinite conjugacy class of $l$. If $l=q=I$, then $d\ne0$ and conjugation by $\SL_3(R)$ produces the infinite orbit of Lemma~\ref{lem:D-infinite-orbits}.

It remains to take $l=I$ and $q\ne I$. The natural action of $Q=\Sp_4(k)$ on $V^*=(k^4)^*$ is faithful, so choose $\varphi\in V^*$ with
$$
 \psi=(1-q)\varphi\ne0.
$$
For $a\in A$, let $x_a=(a\otimes\varphi,0)\in D_i=(A\otimes V^*)\oplus C$. The $C$-coordinate is zero, so in both actions
$$
 (1-\theta_i(1,q))x_a=(a\otimes\psi,0).
$$
Write $\gamma=(d;(1,q))$. Using the multiplication rule in the semidirect product, conjugation by $(x_a;1)$ gives
$$
\begin{aligned}
(x_a;1)\gamma(x_a;1)^{-1}
 &=(x_a;1)(d;(1,q))(-x_a;1)\\
 &=\bigl(x_a+d;(1,q)\bigr)(-x_a;1)\\
 &=\bigl(x_a+d+\theta_i(1,q)(-x_a);(1,q)\bigr)\\
 &=\bigl(d+x_a-\theta_i(1,q)x_a;(1,q)\bigr)\\
 &=\bigl(d+(a\otimes\psi,0);(1,q)\bigr).
\end{aligned}
$$
Hence different values of $a$ give different conjugates. Since $A=R^3$ is infinite, the family
$$
\left\{
(x_a;1)\gamma(x_a;1)^{-1}
\ \middle|\
a\in A
\right\}
$$
is infinite. Therefore the conjugacy class of $\gamma$ is infinite.
\end{proof}

\section{The groups are not isomorphic}

\begin{lemma}\label{lem:linear-no-abelian-normal}
The group $\SL_3(R)$ has no nontrivial abelian normal subgroup.
\end{lemma}

\begin{proof}
Let $B\triangleleft\SL_3(R)$ be abelian and let $b\in B$. Fix $r,s\in \{1,2,3\}$ with $r\ne s$, and define $u_f=I+fE_{rs}$. Since $u_f^{-1}=u_f$ in characteristic two, normality and commutativity imply
$$
 b(u_fbu_f)=(u_fbu_f)b.
$$
Expanding gives
\begin{equation}
 fX+f^2Y=0\qquad(f\in R),
 \label{eq:normal-expansion}
\end{equation}
where
$$
 X=b^2E_{rs}+E_{rs}b^2,
 \qquad
 Y=bE_{rs}bE_{rs}+E_{rs}bE_{rs}b.
$$
Taking $f=1$ and $f=t$ in \eqref{eq:normal-expansion}, we obtain $$(t^2+t)Y=t(X+Y)+(tX+t^2Y)=0.$$ Hence $Y=0$ and then $X=0$.

Since $E_{rs}bE_{rs}=b_{sr}E_{rs}$, the identity $Y=0$ becomes
$$
 b_{sr}(bE_{rs}+E_{rs}b)=0.
$$
If $b_{sr}\ne0$, the domain property gives $bE_{rs}+E_{rs}b=0$, but its $(s,s)$-entry is $b_{sr}$, a contradiction. Hence every off-diagonal entry of $b$ is zero. Write $ b=\operatorname{diag}(\lambda_1,\lambda_2,\lambda_3). $ Since $b\in\SL_3(R)$, $ \lambda_1\lambda_2\lambda_3=\det b=1. $ Thus each $\lambda_i$ is a unit of $R$. Since the only unit of $R=(\mathbb Z/2\mathbb Z)[t]$ is $1$, we obtain $b=I$.
\end{proof}

Recall that an \textbf{elementary abelian group} is an abelian group in which all elements other than the identity have the same order.
\begin{lemma}\label{lem:Q-no-elementary-normal}
The group $Q=\Sp_4(k)$ has no nontrivial normal elementary abelian subgroup of exponent two.
\end{lemma}

\begin{proof}
We first show that $Q$ acts transitively on $V\setminus\{0\}$. Let $v,w\in V\setminus\{0\}$. By nondegeneracy of the symplectic form, both $v$ and $w$ can be completed to symplectic bases of $V$. The linear map sending the first symplectic basis to the second preserves the symplectic form and hence belongs to $Q$. In particular, it sends $v$ to $w$. Hence the natural $Q$-module $V$ is simple, and its action is faithful.

Now let $L\triangleleft Q$ be an elementary abelian subgroup of exponent two. Since $L$ is a finite abelian group of exponent two, we have $L\cong(\Z/2\Z)^n$ for some $n\geq 0$. Since $V=k^4=(\Z/2\Z)^4$,
$$
|V\setminus\{0\}|=2^4-1=15.
$$
The group $L$ acts on this set, and every orbit has cardinality
$$
|L\cdot v|=[L:L_v],
$$
which is a power of two since $|L|=2^n$. If there were no fixed nonzero vector, every orbit would have even cardinality, contradicting the fact that their disjoint union has cardinality $15$. Hence there exists $0\ne v\in V$ such that $lv=v$ for any $l\in L$. Equivalently,
$$
V^L:=\{v\in V\mid lv=v\text{ for every }l\in L\}\ne0.
$$

We claim that $V^L$ is $Q$-invariant. Indeed, if $v\in V^L$, $q\in Q$, and $l\in L$, then by normality of $L$, $q^{-1}lq\in L$, and therefore
$$
l(qv)
=
q(q^{-1}lq)v
=
qv.
$$
Thus $qv\in V^L$. Since $V^L$ is a nonzero $Q$-invariant subspace and $V$ is simple, we have
$$
V^L=V.
$$
Consequently, every element of $L$ acts trivially on $V$. The action of $Q$ on $V$ is faithful, so $L=\{I\}$.
\end{proof}

\begin{proposition}\label{prop:characteristic}
The subgroup $D_i$ is \textbf{characteristic} in $\Gamma_i$, i.e., $\alpha(D_i)=D_i$ for any $\alpha\in\operatorname{Aut}(\Gamma_i)$. In the quotient
$$
 \Gamma_i/D_i\cong\SL_3(R)\times Q,
$$
the subgroup $Q$ is characteristic.
\end{proposition}

\begin{proof}
Let $E\triangleleft\Gamma_i$ be elementary abelian of exponent two. Its image in $\SL_3(R)\times Q$ has abelian normal projection to $\SL_3(R)$ and normal elementary abelian projection to $Q$. Both projections are trivial by Lemmas~\ref{lem:linear-no-abelian-normal} and \ref{lem:Q-no-elementary-normal}. Hence $E\subseteq D_i$. Since $D_i$ itself is also elementary abelian of exponent two, it is the unique largest such normal subgroup and is therefore preserved by every $\alpha\in\operatorname{Aut}(\Gamma_i)$. Hence $D_i$ is characteristic in $\Gamma_i$.

Lemma~\ref{lem:linear-icc} implies that $\SL_3(R)$ has no nontrivial finite normal subgroup. Every finite normal subgroup of $\SL_3(R)\times Q$ therefore projects trivially to $\SL_3(R)$ and is contained in $Q$. Thus $Q$ is the unique largest finite normal subgroup of the quotient and is characteristic.
\end{proof}

Recall that, for a $Q$-module $M$, a map $c:Q\to M$, $q\mapsto c_q$, is a \textbf{$1$-cocycle} if $c_{pq}=p\cdot c_q+c_p$ for all $p,q\in Q$. It is a \textbf{$1$-coboundary} if there exists $m\in M$ such that $c_q=q\cdot m-m$ for every $q\in Q$. Recall from Lemma~\ref{lem:finite-cocycle} that, for $v=(a_1,b_1,a_2,b_2)\in V=k^4$ and $r_0(v)=a_1b_1+a_2b_2$, the map $q\mapsto\ell_q:=q\cdot r_0-r_0$ is a $1$-cocycle with values in $V^*$.
\begin{lemma}\label{lem:ell-nonprincipal}
The cocycle $\ell_q=q\cdot r_0-r_0$ is not a coboundary.
\end{lemma}

\begin{proof}
Suppose $\ell_q=q\cdot\lambda-\lambda$ for some $\lambda\in V^*$ and every $q\in Q$. Then $q\cdot\lambda-\lambda=q\cdot r_0-r_0$, i.e., $q\cdot(r_0-\lambda)=r_0-\lambda$. Hence $r_0-\lambda$ is $Q$-invariant and vanishes at zero. Transitivity of the $Q$-action on $V\setminus\{0\}$ implies that $r_0-\lambda=0$ or $\mathbf 1_{V\setminus\{0\}}$. Recall that the polarization of a function $r:V\to k$ satisfying $r(0)=0$ is
$$
B_r(v,w)=r(v+w)-r(v)-r(w).
$$
It measures the failure of $r$ to be additive. Since $\lambda$ is linear, $B_\lambda=0$, and hence
$$
B_{r_0-\lambda}=B_{r_0}-B_\lambda=\omega.
$$

The zero function has zero polarization, so $r_0-\lambda$ cannot be identically zero. For the second possibility, suppose that $r_0-\lambda=\mathbf 1_{V\setminus\{0\}}$. Choose distinct nonzero vectors $u,w\in V$ such that $\omega(u,w)=0$; for example, one may take $u=e_1$ and $w=e_2$ in a symplectic basis $(e_1,f_1,e_2,f_2)$. Since $u\ne w$ and $\operatorname{char}(k)=2$, we have $u+w\ne0$. Therefore
$$
\begin{aligned}
B_{\mathbf 1_{V\setminus\{0\}}}(u,w)
&=
\mathbf 1_{V\setminus\{0\}}(u+w)
-\mathbf 1_{V\setminus\{0\}}(u)
-\mathbf 1_{V\setminus\{0\}}(w)\\
&=1-1-1=1.
\end{aligned}
$$
This contradicts
$$
B_{r_0-\lambda}(u,w)=\omega(u,w)=0.
$$
Thus neither possibility can occur.
\end{proof}
Recall that a $Q$-module $M$ is \textbf{semisimple} if and only if every $Q$-submodule $N\subseteq M$ admits a $Q$-invariant complement $L$ with $M\cong N\oplus L$ as $Q$-modules; equivalently, every short exact sequence
$$0\longrightarrow N\longrightarrow M\longrightarrow M/N\longrightarrow 0$$ \textbf{splits} $Q$-equivariantly, i.e., there exists a
$Q$-module homomorphism $s:M/N\to M$ such that $\pi_{M/N}\circ s=\mathrm{id}_{M/N}$.

Define the $Q$-module
\begin{equation}\label{eq:Eell}
 E_\ell=V^*\oplus k,
 \qquad q\cdot(\varphi,s)=(q\varphi+s\ell_q,s).
\end{equation}
It fits into
\begin{equation}\label{eq:Eell-extension}
 0\longrightarrow V^*\longrightarrow E_\ell
 \longrightarrow k\longrightarrow0.
\end{equation}
Any $k$-linear splitting $\sigma:k\to E_\ell$ has the form $\sigma(s)=(s\lambda,s)$ for some $\lambda\in V^*$. Since $Q$ acts trivially on $k$, the splitting is $Q$-equivariant if and only if
$$
(q\cdot\lambda+\ell_q,1)
=q\cdot\sigma(1)
=\sigma(1)
=(\lambda,1)
$$
for every $q\in Q$. This is equivalent to $\ell_q=q\cdot\lambda-\lambda$. Thus the extension \eqref{eq:Eell-extension} splits if and only if $\ell$ is a coboundary. Lemma~\ref{lem:ell-nonprincipal} therefore shows that the extension is nonsplit. In particular, $E_\ell$ is not semisimple as a $Q$-module.

\begin{proposition}\label{prop:nonisomorphic}
The groups $\Gamma_1$ and $\Gamma_2$ are not isomorphic.
\end{proposition}

\begin{proof}
As a $Q$-module,
$$
 D_1=(A\otimes V^*)\oplus C,
$$
where $Q$ acts naturally on $V^*$ and trivially on $A$ and $C$. The symplectic form identifies $V$ and $V^*$ equivariantly, so $V^*$ is simple as a $Q$-module. After choosing bases of $A$ and $C$, the module $D_1$ is an algebraic direct sum of copies of the simple modules $V^*$ and $k$. Thus $D_1$ is semisimple as a $Q$-module.

Choose a $k$-linear section $s:A\to C$ of $\delta$. No $\SL_3(R)$-equivariance is needed. Then $C=\ker(\delta)\oplus s(A)$ as a $k$-module. Since $Q$ acts trivially on $\ker(\delta)\subset D_2$, the action $\theta_2$ gives a $Q$-module isomorphism
\begin{equation}\label{eq:D2-decomposition}
 D_2\cong(A\otimes E_\ell)\oplus\ker(\delta).
\end{equation}
Indeed, under $C=\ker(\delta)\oplus s(A)$, the $Q$-action on $(A\otimes V^*)\oplus s(A)\subset (A\otimes V^*)\oplus C=D_2$ is
$$
 q\cdot(b\otimes \varphi ,s(a))=(q(b\otimes \varphi )+\delta(s(a))\otimes\ell_q,s(a))=(b\otimes q \varphi +a\otimes\ell_q,s(a)).
$$
Hence $(A\otimes V^*)\oplus s(A)$ is $Q$-invariant. Moreover, through the map
$$(b\otimes \varphi ,s(a))\in (A\otimes V^*)\oplus s(A) \mapsto b\otimes (\varphi ,0)+a\otimes (0,1)\in A\otimes (V^*\oplus k),$$
$(A\otimes V^*)\oplus s(A)$ is $Q$-equivariantly isomorphic to the tensor product $A\otimes (V^*\oplus k)=A\otimes E_\ell$ of the trivial $Q$-module $A$ and the $Q$-module $E_\ell=(V^*\oplus k)$ from \eqref{eq:Eell}. Hence
$$
D_2\cong (A\otimes V^*)\oplus (s(A)\oplus\ker(\delta) )\cong  (A\otimes E_\ell)\oplus\ker(\delta).
$$

After choosing a basis of $A$, the nonsplit $Q$-module $E_\ell$ occurs as a direct summand of the $Q$-module $D_2$. If $D_2$ were semisimple, then its direct summand $E_\ell$ would be semisimple, contradicting the nonsplitting of \eqref{eq:Eell-extension}. Therefore $D_2$ is not semisimple.

Suppose that $\alpha:\Gamma_1\to\Gamma_2$ is an isomorphism. By Proposition~\ref{prop:characteristic}, $D_i$ is the unique largest normal elementary abelian subgroup of exponent two in $\Gamma_i$. Since this property is preserved under group isomorphisms,
$$
\alpha(D_1)=D_2.
$$
Hence $\alpha$ induces an isomorphism
$$
\overline{\alpha}:\Gamma_1/D_1\longrightarrow\Gamma_2/D_2.
$$
Using the canonical identifications
$$
\Gamma_i/D_i\cong\SL_3(R)\times Q,
$$
we regard $\overline{\alpha}$ as an automorphism of $\SL_3(R)\times Q$. Since $Q$ is characteristic in this quotient, we have $\overline{\alpha}(Q)=Q$. Thus $\beta:=\overline{\alpha}|_Q \in \operatorname{Aut}(Q)$.

Let $T=\alpha|_{D_1}:D_1\to D_2$. Since $D_1$ and $D_2$ are $k$-vector spaces with $k=\Z/2\Z$, the group isomorphism $T$ is automatically $k$-linear. For $q\in Q$, let
$$
\widetilde q=(0;(1,q))\in\Gamma_1=D_1\rtimes(\SL_3(R)\times Q).
$$
Since the image of $\alpha(\widetilde q)$ in the quotient is $(1,\beta(q))$, there exists $c_q\in D_2$ such that
$$
\alpha(\widetilde q)=(c_q;(1,\beta(q))).
$$
For every $d\in D_1$, we therefore have
$$
\begin{aligned}
T\bigl(\theta_1(1,q)d\bigr)
&=\alpha\bigl(\widetilde q\,d\,\widetilde q^{-1}\bigr)\\
&=\alpha(\widetilde q)\,T(d)\,\alpha(\widetilde q)^{-1}\\
&=\theta_2(1,\beta(q))T(d).
\end{aligned}
$$
In the last equality, the factor $c_q\in D_2$ has no effect because $D_2$ is abelian, so conjugation by $c_q$ acts trivially on $D_2$.

Let $\beta^*D_2$ denote the pullback of the $Q$-module $D_2$ along $\beta$, namely the module $D_2$ with the $Q$-action
$$
q\ast d=\theta_2(1,\beta(q))d.
$$
The preceding identity says precisely that
$$
T:D_1\longrightarrow\beta^*D_2
$$
is a $Q$-module isomorphism. Since $\beta$ is an automorphism of $Q$, the modules $D_2$ and $\beta^*D_2$ have exactly the same invariant subspaces and are therefore semisimple simultaneously. It follows that the semisimplicity of $D_1$ would imply the semisimplicity of $D_2$, contradicting the semisimple--nonsemisimple distinction established above. Hence $\Gamma_1\not\cong\Gamma_2$.
\end{proof}

\section{Completion of the proof}

\begin{proof}[Proof of Theorem~\ref{thm:main}]
The explicit semidirect products in \eqref{eq:groups} are countable. Proposition~\ref{prop:Gamma-T} proves property~(T), and Proposition~\ref{prop:Gamma-icc} proves ICC. By Proposition~\ref{prop:nonisomorphic}, $\Gamma_1\not\cong\Gamma_2$, while Proposition~\ref{prop:factor-iso} supplies a $*$-isomorphism between their group von Neumann algebras. All four assertions of the theorem follow.
\end{proof}

\bibliographystyle{alphaurl}
\bibliography{ref2}

@book {Mar91,
    AUTHOR = {Margulis, G. A.},
     TITLE = {Discrete subgroups of semisimple {L}ie groups},
    SERIES = {Ergebnisse der Mathematik und ihrer Grenzgebiete (3) [Results
              in Mathematics and Related Areas (3)]},
    VOLUME = {17},
 PUBLISHER = {Springer-Verlag, Berlin},
      YEAR = {1991},
     PAGES = {x+388},
      ISBN = {3-540-12179-X},
   MRCLASS = {22E40 (20Hxx 22-02 22D40)},
  MRNUMBER = {1090825},
MRREVIEWER = {Gopal\ Prasad},
       DOI = {10.1007/978-3-642-51445-6},
       URL = {https://doi.org/10.1007/978-3-642-51445-6},
}

@article {EJK,
    AUTHOR = {Ershov, Mikhail and Jaikin-Zapirain, Andrei and Kassabov,
              Martin},
     TITLE = {Property {$(T)$} for groups graded by root systems},
   JOURNAL = {Mem. Amer. Math. Soc.},
  FJOURNAL = {Memoirs of the American Mathematical Society},
    VOLUME = {249},
      YEAR = {2017},
    NUMBER = {1186},
     PAGES = {v+135},
      ISSN = {0065-9266,1947-6221},
      ISBN = {978-1-4704-2604-0; 978-1-4704-4139-5},
   MRCLASS = {22D10 (17B22 17B70 20E42)},
  MRNUMBER = {3724373},
MRREVIEWER = {Dave\ Witte\ Morris},
       DOI = {10.1090/memo/1186},
       URL = {https://doi.org/10.1090/memo/1186},
}

@article {IPV10,
    AUTHOR = {Ioana, Adrian and Popa, Sorin and Vaes, Stefaan},
     TITLE = {A class of superrigid group von {N}eumann algebras},
   JOURNAL = {Ann. of Math. (2)},
  FJOURNAL = {Annals of Mathematics. Second Series},
    VOLUME = {178},
      YEAR = {2013},
    NUMBER = {1},
     PAGES = {231--286},
      ISSN = {0003-486X,1939-8980},
   MRCLASS = {46L10 (20F99)},
  MRNUMBER = {3043581},
       DOI = {10.4007/annals.2013.178.1.4},
       URL = {https://doi.org/10.4007/annals.2013.178.1.4},
}

@article {CIOS21,
    AUTHOR = {Chifan, Ionu\c{t} and Ioana, Adrian and Osin, Denis and Sun,
              Bin},
     TITLE = {Wreath-like products of groups and their von {N}eumann
              algebras {I}: {$\rm W^\ast $}-superrigidity},
   JOURNAL = {Ann. of Math. (2)},
  FJOURNAL = {Annals of Mathematics. Second Series},
    VOLUME = {198},
      YEAR = {2023},
    NUMBER = {3},
     PAGES = {1261--1303},
      ISSN = {0003-486X,1939-8980},
   MRCLASS = {46L10 (20E22 20F67 22D25 22D55 46L36)},
  MRNUMBER = {4660139},
MRREVIEWER = {\`E.\ V.\ Kissin},
       DOI = {10.4007/annals.2023.198.3.6},
       URL = {https://doi.org/10.4007/annals.2023.198.3.6},
}

@incollection{Co82,
  author    = {Connes, Alain},
  title     = {Classification des facteurs},
  booktitle = {Operator Algebras and Applications, Part 2 (Kingston, Ont., 1980)},
  series    = {Proceedings of Symposia in Pure Mathematics},
  volume    = {38},
  pages     = {43--109},
  publisher = {American Mathematical Society},
  address   = {Providence, RI},
  year      = {1982},
       URL = {https://bookstore.ams.org/view?ProductCode=PSPUM%2F38.2}
}

@article{BR95,
  author  = {Bates, Teresa and Robertson, Guyan},
  title   = {Positive definite functions and relative property {(T)}
             for subgroups of discrete groups},
  journal = {Bulletin of the Australian Mathematical Society},
  volume  = {52},
  number  = {1},
  pages   = {31--39},
  year    = {1995},
  doi     = {10.1017/S000497270001443X},
       URL = {https://doi.org/10.1017/S000497270001443X}
}

@misc{danus,
      title={Danus: Orchestrating Mathematical Reasoning Agents with Fact-Graph Memory}, 
      author={Jihao Liu and Guoxiong Gao and Zeming Sun and Bin Wu and Shurui Liu and Jiedong Jiang and Haocheng Ju and Leheng Chen and Ronnie Cheng and Xiping Zhang and Bin Dong},
      year={2026},
      eprint={2607.06447},
      archivePrefix={arXiv},
      primaryClass={cs.AI},
       URL = {https://arxiv.org/abs/2607.06447},
}

@book {Mostow73,
    AUTHOR = {Mostow, G. D.},
     TITLE = {Strong rigidity of locally symmetric spaces},
    SERIES = {Annals of Mathematics Studies},
    VOLUME = {No. 78},
 PUBLISHER = {Princeton University Press, Princeton, NJ; University of Tokyo
              Press, Tokyo},
      YEAR = {1973},
     PAGES = {v+195},
   MRCLASS = {22E40 (53C35)},
  MRNUMBER = {385004},
MRREVIEWER = {M.\ S.\ Raghunathan},
       URL = {https://doi.org/10.1515/9781400881833}
}

@misc{OpenAI26,
  author = {{OpenAI}},
  title = {Ten Advances in Mathematics and Theoretical Computer Science},
  year = {2026},
  month = aug,
  note = {Chapter 4: A Counterexample to Connes's Rigidity Conjecture},
       URL = {https://cdn.openai.com/pdf/ten-proofs-oai.pdf},
}

@article {Con80,
    AUTHOR = {Connes, A.},
     TITLE = {A factor of type {${\rm II}\sb{1}$}\ with countable
              fundamental group},
   JOURNAL = {J. Operator Theory},
  FJOURNAL = {Journal of Operator Theory},
    VOLUME = {4},
      YEAR = {1980},
    NUMBER = {1},
     PAGES = {151--153},
      ISSN = {0379-4024},
   MRCLASS = {46L10 (22D35)},
  MRNUMBER = {587372},
MRREVIEWER = {Richard\ I.\ Loebl},
       URL = {https://jot.theta.ro/jot/archive/1980-004-001/1980-004-001-008.pdf}
}

@article {CIOSII,
    AUTHOR = {Chifan, Ionu\c{t} and Ioana, Adrian and Osin, Denis and Sun,
              Bin},
     TITLE = {Wreath-like products of groups and their von {N}eumann
              algebras {II}: outer automorphisms},
   JOURNAL = {Duke Math. J.},
  FJOURNAL = {Duke Mathematical Journal},
    VOLUME = {175},
      YEAR = {2026},
    NUMBER = {2},
     PAGES = {287--359},
      ISSN = {0012-7094,1547-7398},
   MRCLASS = {46L10 (20E22 20F67 22D25 22D55 46L36)},
  MRNUMBER = {5030466},
       DOI = {10.1215/00127094-2025-0028},
       URL = {https://doi.org/10.1215/00127094-2025-0028},
}

@article {CIOSIII,
    AUTHOR = {Chifan, Ionu\c{t} and Ioana, Adrian and Osin, Denis and Sun,
              Bin},
     TITLE = {Wreath-like products of groups and their von {N}eumann
              algebras {III}: embeddings},
   JOURNAL = {Comm. Math. Phys.},
  FJOURNAL = {Communications in Mathematical Physics},
    VOLUME = {407},
      YEAR = {2026},
    NUMBER = {2},
     PAGES = {Paper No. 38, 24},
      ISSN = {0010-3616,1432-0916},
   MRCLASS = {46L36 (20E22 20E45 46L10)},
  MRNUMBER = {5017847},
       DOI = {10.1007/s00220-025-05508-x},
       URL = {https://doi.org/10.1007/s00220-025-05508-x},
}

@article {CFOT26,
    AUTHOR = {Chifan, Ionu\c{t} and Fern\'andez Quero, Adriana and Osin,
              Denis and Tan, Hui},
     TITLE = {{$\rm W^*$}-superrigidity for property ({T}) groups with
              infinite center},
   JOURNAL = {Adv. Math.},
  FJOURNAL = {Advances in Mathematics},
    VOLUME = {496},
      YEAR = {2026},
     PAGES = {Paper No. 110979, 39},
      ISSN = {0001-8708,1090-2082},
   MRCLASS = {46L10 (20E22 22D25 46L05 46L35)},
  MRNUMBER = {5064416},
       DOI = {10.1016/j.aim.2026.110979},
       URL = {https://doi.org/10.1016/j.aim.2026.110979},
}

@article {DV25a,
    AUTHOR = {Donvil, Milan and Vaes, Stefaan},
     TITLE = {{${\rm W}^*$}-superrigidity for cocycle twisted group von
              {N}eumann algebras},
   JOURNAL = {Invent. Math.},
  FJOURNAL = {Inventiones Mathematicae},
    VOLUME = {240},
      YEAR = {2025},
    NUMBER = {1},
     PAGES = {193--260},
      ISSN = {0020-9910,1432-1297},
   MRCLASS = {46L10 (20E22 20E34 20F65)},
  MRNUMBER = {4871958},
MRREVIEWER = {Paul\ Jolissaint},
       DOI = {10.1007/s00222-025-01320-5},
       URL = {https://doi.org/10.1007/s00222-025-01320-5},
}

@article {DV25b,
    AUTHOR = {Donvil, Milan and Vaes, Stefaan},
     TITLE = {{${\rm W}^*$}-superrigidity for groups with infinite center},
   JOURNAL = {Adv. Math.},
  FJOURNAL = {Advances in Mathematics},
    VOLUME = {480},
      YEAR = {2025},
     PAGES = {Paper No. 110527, 63},
      ISSN = {0001-8708,1090-2082},
   MRCLASS = {46L36 (46L10)},
  MRNUMBER = {4957547},
MRREVIEWER = {David\ Gao},
       DOI = {10.1016/j.aim.2025.110527},
       URL = {https://doi.org/10.1016/j.aim.2025.110527},
}

@article {Io11,
    AUTHOR = {Ioana, Adrian},
     TITLE = {{$W^*$}-superrigidity for {B}ernoulli actions of property
              ({T}) groups},
   JOURNAL = {J. Amer. Math. Soc.},
  FJOURNAL = {Journal of the American Mathematical Society},
    VOLUME = {24},
      YEAR = {2011},
    NUMBER = {4},
     PAGES = {1175--1226},
      ISSN = {0894-0347,1088-6834},
   MRCLASS = {46L36 (28D15 37A20)},
  MRNUMBER = {2813341},
MRREVIEWER = {Stefaan\ Vaes},
       DOI = {10.1090/S0894-0347-2011-00706-6},
       URL = {https://doi.org/10.1090/S0894-0347-2011-00706-6},
}

@article {CH89,
    AUTHOR = {Cowling, Michael and Haagerup, Uffe},
     TITLE = {Completely bounded multipliers of the {F}ourier algebra of a
              simple {L}ie group of real rank one},
   JOURNAL = {Invent. Math.},
  FJOURNAL = {Inventiones Mathematicae},
    VOLUME = {96},
      YEAR = {1989},
    NUMBER = {3},
     PAGES = {507--549},
      ISSN = {0020-9910,1432-1297},
   MRCLASS = {22E27 (22E40 43A22 43A80 46J99)},
  MRNUMBER = {996553},
MRREVIEWER = {Christopher\ Meaney},
       DOI = {10.1007/BF01393695},
       URL = {https://doi.org/10.1007/BF01393695},
}

@article {Po06a,
    AUTHOR = {Popa, Sorin},
     TITLE = {Strong rigidity of {$\rm II_1$} factors arising from malleable
              actions of {$w$}-rigid groups. {I}},
   JOURNAL = {Invent. Math.},
  FJOURNAL = {Inventiones Mathematicae},
    VOLUME = {165},
      YEAR = {2006},
    NUMBER = {2},
     PAGES = {369--408},
      ISSN = {0020-9910,1432-1297},
   MRCLASS = {46L10 (22D25 37A20 46L55)},
  MRNUMBER = {2231961},
MRREVIEWER = {Alain\ Valette},
       DOI = {10.1007/s00222-006-0501-4},
       URL = {https://doi.org/10.1007/s00222-006-0501-4},
}

@article {Po06b,
    AUTHOR = {Popa, Sorin},
     TITLE = {Strong rigidity of {$\rm II_1$} factors arising from malleable
              actions of {$w$}-rigid groups. {II}},
   JOURNAL = {Invent. Math.},
  FJOURNAL = {Inventiones Mathematicae},
    VOLUME = {165},
      YEAR = {2006},
    NUMBER = {2},
     PAGES = {409--451},
      ISSN = {0020-9910,1432-1297},
   MRCLASS = {46L55 (22D25 37A20 37A55 46L10)},
  MRNUMBER = {2231962},
MRREVIEWER = {Alain\ Valette},
       DOI = {10.1007/s00222-006-0502-3},
       URL = {https://doi.org/10.1007/s00222-006-0502-3},
}

@misc{CFT24,
      title={Rigidity results for group von Neumann algebras with diffuse center}, 
      author={Chifan, Ionu\c{t} and Fern\'andez Quero, Adriana and Tan, Hui},
      year={2024},
      eprint={2403.01280},
      archivePrefix={arXiv},
      primaryClass={math.OA},
       URL = {https://arxiv.org/abs/2403.01280},
}

@incollection {HouICM,
    AUTHOR = {Houdayer, Cyril},
     TITLE = {Noncommutative ergodic theory of higher rank lattices},
 BOOKTITLE = {I{CM}---{I}nternational {C}ongress of {M}athematicians. {V}ol.
              4. {S}ections 5--8},
     PAGES = {3202--3223},
 PUBLISHER = {EMS Press, Berlin},
      YEAR = {[2023] \copyright 2023},
      ISBN = {978-3-98547-062-4; 978-3-98547-562-9; 978-3-98547-058-7},
   MRCLASS = {22D25 (22E40 37B05 46L10 46L55)},
  MRNUMBER = {4680358},
MRREVIEWER = {Judith\ A.\ Packer},
       DOI = {10.4171/ICM2022/39},
       URL = {https://doi.org/10.4171/ICM2022/39},
}

@book {BHV08,
    AUTHOR = {Bekka, Bachir and de la Harpe, Pierre and Valette, Alain},
     TITLE = {Kazhdan's property ({T})},
    SERIES = {New Mathematical Monographs},
    VOLUME = {11},
 PUBLISHER = {Cambridge University Press, Cambridge},
      YEAR = {2008},
     PAGES = {xiv+472},
      ISBN = {978-0-521-88720-5},
   MRCLASS = {22-02 (22E40 28D15 37A15 43A07 43A35)},
  MRNUMBER = {2415834},
MRREVIEWER = {Markus\ Neuhauser},
       DOI = {10.1017/CBO9780511542749},
       URL = {https://doi.org/10.1017/CBO9780511542749},
}

@article{EJ,
  author       = {Ershov, Mikhail and Jaikin-Zapirain, Andrei},
  title        = {Property ({T}) for noncommutative universal lattices},
  journal      = {Invent. Math.},
  fjournal     = {Inventiones Mathematicae},
  volume       = {179},
  year         = {2010},
  number       = {2},
  pages        = {303--347},
  issn         = {0020-9910,1432-1297},
  mrclass      = {22D10 (20E08 20E18 20F69)},
  mrnumber     = {2570119},
  mrreviewer   = {Markus Neuhauser},
  doi          = {10.1007/s00222-009-0218-2},
  url          = {https://doi.org/10.1007/s00222-009-0218-2},
}

@article {Sus77,
    AUTHOR = {Suslin, A. A.},
     TITLE = {The structure of the special linear group over rings of
              polynomials},
   JOURNAL = {Izv. Akad. Nauk SSSR Ser. Mat.},
  FJOURNAL = {Izvestiya Akademii Nauk SSSR. Seriya Matematicheskaya},
    VOLUME = {41},
      YEAR = {1977},
    NUMBER = {2},
     PAGES = {235--252, 477},
      ISSN = {0373-2436},
   MRCLASS = {13C10 (14F05)},
  MRNUMBER = {472792},
MRREVIEWER = {Chr.\ U.\ Jensen},
}

@article {Sha99,
    AUTHOR = {Shalom, Yehuda},
     TITLE = {Bounded generation and {K}azhdan's property ({T})},
   JOURNAL = {Inst. Hautes \'Etudes Sci. Publ. Math.},
  FJOURNAL = {Institut des Hautes \'Etudes Scientifiques. Publications
              Math\'ematiques},
    NUMBER = {90},
      YEAR = {1999},
     PAGES = {145--168},
      ISSN = {0073-8301,1618-1913},
   MRCLASS = {22E46 (22E67)},
  MRNUMBER = {1813225},
MRREVIEWER = {Paul\ Jolissaint},
       URL = {http://www.numdam.org.proxy.rubens.ens.fr/item?id=PMIHES_1999__90__145_0},
}

@incollection {Sha06,
    AUTHOR = {Shalom, Yehuda},
     TITLE = {The algebraization of {K}azhdan's property ({T})},
 BOOKTITLE = {International {C}ongress of {M}athematicians. {V}ol. {II}},
     PAGES = {1283--1310},
 PUBLISHER = {Eur. Math. Soc., Z\"urich},
      YEAR = {2006},
      ISBN = {978-3-03719-022-7},
   MRCLASS = {22D10 (16P60 19B10 20D06 20G05)},
  MRNUMBER = {2275645},
MRREVIEWER = {Alain\ Valette},
}

\end{document}